\documentclass[10pt]{amsart}
\usepackage{amsmath}
\usepackage[usenames,dvipsnames]{color}
\usepackage{parskip}
\usepackage{hyperref}
\usepackage{amsfonts}
\usepackage{amscd}
\usepackage[centertags]{amsmath}
\usepackage{amssymb}
\usepackage[all,cmtip]{xy}
\usepackage[english]{babel}
\usepackage{tabularx}
\usepackage{mathtools}
\usepackage{amsxtra}
\usepackage{euscript}
\usepackage[T1]{fontenc}
\usepackage{doc, exscale, fontenc, latexsym, syntonly}
\usepackage{amsfonts}
\usepackage{amsthm}
\usepackage{graphicx}
\usepackage{xcolor}
\usepackage{tikz-cd}

\numberwithin{figure}{section}
\numberwithin{table}{section}

\newcommand{\cat}{\ensuremath{\mathrm{cat}}}
\newcommand{\scat}{\ensuremath{\mathrm{scat}}}
\newcommand{\secat}{\ensuremath{\mathrm{secat}}}
\newcommand{\hsecat}{\ensuremath{\mathrm{hsecat}}}
\newcommand{\id}{\ensuremath{\mathrm{id}}}
\newcommand{\TC}{\ensuremath{\mathrm{TC}}}
\newcommand{\D}{\ensuremath{\mathrm{D}}}
\newcommand{\pr}{\ensuremath{\mathrm{pr}}}
\newcommand{\SD}{\ensuremath{\mathrm{SD}}}
\newcommand{\sd}{\ensuremath{\mathrm{sd}}}

\newtheorem{theorem}{Theorem}[section]
\newtheorem{definition}{Definition}[section]
\newtheorem{corollary}{Corollary}[section]
\newtheorem{remark}{Remark}[section]
\newtheorem{example}{Example}[section]
\newtheorem{proposition}{Proposition}[section]
\newtheorem{lemma}{Lemma}[section]

\title{$m-$Contiguity Distance}

\author{N\.{I}lay Ek\.{I}z Yazici\textsuperscript{1}, Nursultan Kuanyshov, Ayşe Borat\textsuperscript{2}}

\date{\today}



\address{\textsc{Nilay Ekiz Yazıcı}
Bursa Technical University\\
Faculty of Engineering and Natural Sciences\\
Department of Mathematics\\
Bursa, Turkiye}
\email{nilay.ekiz@btu.edu.tr}

\address{\textsc{Nursultan Kuanyshov}
Suleyman Demirel University (SDU)\\
School of Information Technologies and Applied Mathematics (SITAM)\\
Department of Mathematics\\
Kaskelen, Kazakhistan}
\email{nursultan.kuanyshov@sdu.edu.kz} 

\address{\textsc{Ayşe Borat}
Bursa Technical University\\
Faculty of Engineering and Natural Sciences\\
Department of Mathematics\\
Bursa, Turkiye}
\email{ayse.borat@btu.edu.tr}

\begin{document}

\begin{abstract} In this paper, we systematically develop the $m$-contiguity distance between simplicial maps as a discrete approximation framework for homotopical complexity in the category of simplicial complexes. We construct an increasing sequence of invariants that approximate the contiguity distance from below. We prove that $m$-contiguity distance is invariant under strong homotopy equivalence and that $m$-contiguity distance coincides with the usual contiguity distance provided that the dimension of the domain simplicial complex is $m$. The fundamental properties of $m$-contiguity distance are established, including its behaviour under barycentric subdivision, under compositions, and a categorical product inequality. As applications of this theory, we define the $m$-simplicial Lusternik–Schnirelmann category and the $m$-discrete topological complexity, proving that each arises naturally as a special case of 
$m$-contiguity distance. We also showed that $\SD_1(\varphi,\psi)=\SD(\varphi,\psi)$ under some conditions related to aspherical spaces.
\end{abstract}

\maketitle

\footnotetext[1] {The author was supported by the TÜBİTAK BIDEB 2211-A Domestic Doctoral Scholarship Program.}
\footnotetext[2]{The corresponding author.}

\section{Introduction}

Numerical invariants measuring the homotopical complexity of spaces and maps play a central role in algebraic topology. Classical examples include the Lusternik--Schnirelmann category \cite{LS}, the Schwarz genus of a fibration \cite{Sch}, and Farber's topological complexity \cite{Fa}. These invariants admit interpretations in terms of the minimal number of local homotopical trivializations required to describe global behavior and have been extensively studied in both theoretical and applied contexts.

In this direction, Macias-Virgós and Mosquera-Lois introduced the notion of \emph{homotopic distance} between continuous maps \cite{MVML}. Given two maps with a common domain, the homotopic distance measures the minimal number of subsets covering the domain on which the maps are homotopic. This invariant provides a unifying framework encompassing several classical constructions, including LS-category and topological complexity.

Subsequently, the same authors joint with Oprea introduced the \emph{\(m\)-homotopic distance} \cite{MVMLO}. Unlike homotopic distance, the \(m\)-homotopic distance is not a higher-dimensional generalization but rather a family of invariants that \emph{approximate from below} classical homotopical complexity invariants. More precisely, \(m\)-homotopic distance yields an increasing sequence of numerical invariants converging to homotopic distance as \(m\to\infty\), and similarly provides lower approximations to LS-category, Schwarz genus, and topological complexity. This filtration captures progressively finer homotopical information while remaining more accessible for small values of \(m\). In paper \cite{ADDD}, Arora et al. studied continuous settings for $m$-homotopic distance. 

Parallel to these developments in the continuous setting, discrete models of homotopy theory have attracted increasing attention \cite{ABCD,BPV,FTGCMVV,FTMVMV1,FTMVMV2, FTMVV}. Within the category of simplicial complexes, homotopies are replaced by the notion of \emph{contiguity} between simplicial maps, a classical combinatorial relation compatible with simplicial approximation. Using this framework, Borat, Pamuk and Vergili studied the \emph{contiguity distance} between simplicial maps as a discrete analogue of homotopic distance \cite{BPV}. They established foundational properties of this invariant, including its behaviour under barycentric subdivision, relations with simplicial Lusternik--Schnirelmann category, and product inequalities.

The main objective of this paper is to introduce a discrete approximation framework parallel to that of the \(m\)-homotopic distance. Given simplicial maps between simplicial complexes, we define the \emph{\(m\)-contiguity distance}, which provides a sequence of discrete invariants approximating the contiguity distance from below. In analogy to the continuous setting, the \(m\)-contiguity distance forms an increasing family of invariants whose limit recovers the classical contiguity distance.

We develop the basic theory of the \(m\)-contiguity distance and prove discrete counterparts of the fundamental properties satisfied by the \(m\)-homotopic distance. In particular, we prove that $m$- contiguity distance is invariant under strong homotopy equivalence (Theorem \ref{StrongHI}), analyse its behavior under compositions of simplicial maps, prove inequalities under barycentric subdivision (Theorem \ref{SD-barycentric}), and establish inequalities with respect to categorical products of simplicial complexes (Theorem \ref{categoricalproduct}). Furthermore, we establish a Fox-type phenomenon for the $m$-homotopic distance in the setting of contiguity distance. In particular, we prove that if the dimension of the domain simplicial complex is $m$, then the $m$-contiguity distance coincides with the contiguity distance (Theorem \ref{dimen}).  These results show that \(m\)-contiguity distance fits naturally into the combinatorial homotopy framework.

As applications, we introduce several simplicial invariants arising from this approximation scheme. We define the \emph{\(m\)-simplicial Lusternik-Schnirelmann category} (Theorem \ref{scat}) and the \emph{\(m\)-discrete topological complexity} (Theorem \ref{TC^{m}}), both of which approximate their classical simplicial counterparts from below. Furthermore, using Moore path complexes and simplicial fibrations, we define the \emph{\(m\)-simplicial Schwarz genus} and the \emph{\(m\)-homotopy simplicial Schwarz genus} of a simplicial map (Theorem \ref{hsecat=secat}, Theorem \ref{hsecat}).

\noindent\textbf{Organization of the paper.}
The paper is organised as follows. In Section~2, we recall the necessary background material, including the definitions of contiguity distance, simplicial Lusternik--Schnirelmann category, Moore paths, and the Moore path complex. In Section~3, we introduce the notion of $m$-contiguity distance and establish its fundamental properties. We prove that the $m$-contiguity distance is invariant under strong homotopy equivalence and establish a Fox-type phenomenon for the $m$-contiguity distance. Furthermore, we prove that the $m$-simplicial Lusternik--Schnirelmann category arises as a special case of $m$-contiguity distance. Section~4 is devoted to structural properties of $m$-contiguity distance, where we establish a categorical product formula. In Section~5, we define the $m$-simplicial Švarc genus and the $m$-homotopical simplicial Švarc genus of a simplicial map. We prove that these two invariants coincide whenever the map is a simplicial fibration. In Section~6, we introduce the notion of $m$-simplicial topological complexity and show that it is a special case of $m$-contiguity distance. We also provide a Švarc-type characterisation of $m$-discrete topological complexity. Finally, in Section~7, we extend the argument to aspherical spaces and obtain analogous results in the simplicial setting.    

\section{Preliminary}
 
In this section, we recall the contiguity class of simplicial maps, the simplical LS category, the contiguity distance, Moore Paths, and the Path complex. 

Before proceeding, we should note that throughout this paper, all simplicial complexes are assumed to be edge-path connected.
\subsection{Contiguity class}
\begin{definition}
Two simplicial maps $\varphi, \psi : K \to L$ are called contiguous if for  every simplex $\{v_0, \dots, v_k\}$ in $K$, $\{\varphi(v_0), \dots , \varphi(v_k),\psi(v_0), \dots , \psi(v_k)\}$ constitutes a simplex in $L$. Such maps are denoted by $\varphi \sim_c \psi$.
\end{definition}

\begin{definition}
Two simplicial maps $\varphi, \psi : K \rightarrow L $ are said to be in the same contiguity class if one can find a finite sequence of simplicial maps $\varphi_i : K \rightarrow L$ for $i=0,1,\dots m$ such that $\varphi = \varphi_1 \sim_c \varphi_2 \sim_c \dots \sim_c \varphi_m= \psi$. Such maps are denoted by $\varphi \sim \psi $.
\end{definition}

\begin{definition}\cite{BPV} For simplicial maps $\varphi, \psi:K\rightarrow K'$, the contiguity distance between $\varphi$ and $\psi$, denoted by $\SD(\varphi,\psi)$, is the least integer $k\geq 0$ such that there exists a covering of $K$ by subcomplexes $K_0,K_1,...,K_k$ with the property that $\varphi|_{K_j}$ and $\psi|_{K_j}$ are in the same contiguity class for all $j=0,1,...,k.$ If there is no such covering, it is defined to be $\SD(\varphi,\psi)=\infty$.
\end{definition}

\begin{definition} \cite{FTMVV}
Let $K$ be a simplicial complex and $\Omega\subset K$ be a subcomplex. If the inclusion $i:\Omega\hookrightarrow K$ and a constant simplicial map $c_{v_0}: \Omega\rightarrow K$, where $v_0 \in K$ is some fixed vertex, are in the same contiguity class, then $\Omega$ is called categorical.
\end{definition}

\begin{definition} \cite{FTMVV} Let $K$ be a simplicial complex. The simplicial Lusternik-Schnirelmann category $\scat(K)$ is the least integer $k\geq 0$ such that one can find categorical subcomplexes $K_0,K_1,\dots, K_k$ of $K$ covering $K$.
\end{definition}

\subsection{Moore Paths} 

Let $\textbf{Z}$ be the one--dimensional simplicial complex whose vertex set consists of all integers
$i \in \mathbb{Z}$ and whose $1$--simplices are consecutive the pairs $\{i,i+1\}$. In this way, $\textbf{Z}$
defines a triangulation of the real line.

\begin{definition} \cite{Gr}
Let $K$ be a simplicial complex. A \emph{Moore path} in $K$ is a simplicial map $ \gamma : \textbf{Z} \to K $
which is eventually constant at both ends. That is, there exist integers
$i^{-}, i^{+} \in \mathbb{Z}$ such that
\begin{enumerate}
  \item $\gamma(i) = \gamma(i^{-})$ for all $i \leq i^{-}$,
  \item $\gamma(i) = \gamma(i^{+})$ for all $i \geq i^{+}$.
\end{enumerate}
\end{definition}

If $i^{-} = i^{+}$, then $\gamma$ is a constant map. For a non--constant Moore path
$\gamma \colon \textbf{Z} \to K$, we define
\[
\gamma^{-} := \max \{\, i^{-} \mid \gamma(i) = \gamma(i^{-}) \text{ for all } i \leq i^{-} \,\},
\]
\[
\gamma^{+} := \min \{\, i^{+} \mid \gamma(i) = \gamma(i^{+}) \text{ for all } i \geq i^{+} \,\}.
\]
Clearly, $\gamma^{-} < \gamma^{+}$.

\begin{definition}\cite{FTGCMVV} \label{initialfinal}
The images $\alpha(\gamma):=\gamma(\gamma^{-})$ and $\omega(\gamma):=\gamma(\gamma^{+})$
are called the \emph{initial vertex} and the \emph{final vertex} of $\gamma$, respectively.
If $\gamma$ is constant, we set $\gamma^{-} = 0 = \gamma^{+}$.
\end{definition}

For integers $a,b \in \mathbb{Z}$ with $a \leq b$, let $[a,b]$ denote the full subcomplex of
$\textbf{Z}$ generated by all vertices $i$ such that $a \leq i \leq b$. With this notation, any Moore
path $\gamma$ in $K$ can be identified with the restricted simplicial map
$\gamma \colon [\gamma^{-},\gamma^{+}] \to K.$
The interval $[\gamma^{-}, \gamma^{+}]$ is called the \emph{support} of $\gamma$.

If $\gamma$ is a Moore path in $K$ with support $[\gamma^{-}, \gamma^{+}]$, its \emph{reverse
Moore path} $\overline{\gamma}$ is defined by
\[
\overline{\gamma}(i) := \gamma(-i),
\]
whose support is the interval $[-\gamma^{+}, -\gamma^{-}]$. This reparametrization represents
$\gamma$ traversed in the opposite direction.

If $\gamma$ is a Moore path in $K$ with support $[\gamma^{-}, \gamma^{+}]$ such that
$\gamma^{+} - \gamma^{-} = m$, we define the \emph{normalized Moore path} $|\gamma| \colon I_m \to K$
by
$|\gamma|(i) := \gamma(i + \gamma^{-}).
$
The advantage of this normalization is that the support of $|\gamma|$ is the interval $[0,m]$,
which is more convenient when working with simplicial fibrations.

\begin{definition}\cite{FTGCMVV} 
Let $\gamma$ and $\delta$ be Moore paths in $K$ such that
$\omega(\gamma) = \alpha(\delta).$
The \emph{product path} $\gamma * \delta$ is defined by
$$
(\gamma * \delta)(i) :=
\begin{cases}
\gamma(i - \delta^{-}), & \text{if } i \leq \gamma^{+} + \delta^{-}, \\[4pt]
\delta(i - \gamma^{+}), & \text{if } i \geq \gamma^{+} + \delta^{-}.
\end{cases}
$$
\end{definition}

It follows directly from the definition that the support of $\gamma * \delta$ is $[\gamma^{-}+\delta^{-},\gamma^{+}+\delta^{+}].$
The product of
Moore paths is strictly associative. More precisely, if
$\gamma, \delta$ and $\varepsilon$ are Moore paths satisfying $\omega(\gamma) = \alpha(\delta)
\quad \text{and} \quad
\omega(\delta) = \alpha(\varepsilon),$
then the equality
$\gamma * (\delta * \varepsilon) = (\gamma * \delta) * \varepsilon
$
holds.

Furthermore, if $c_v$ denotes the constant Moore path at a vertex $v \in K$, then one verifies that
$\gamma * c_{\omega} = \gamma = c_{v} * \gamma$ where 
$v = \alpha(\gamma)$ and $\omega=\omega(\gamma).$ 

\subsection{ The path complex } 

We will consider a notion of the Moore path complex associated with a simplicial complex $K$. 

Let $K$ and $L$ be simplicial complexes. We define $L^{K}$ as the simplicial complex whose vertices are all simplicial maps $f \colon K \to L$ and we consider as simplices the finite sets
$\{f_0,\dots,f_p\}$ of simplicial maps $K\to L$ such that $
\bigcup_{i=0}^{p} f_i(\sigma)$
is a simplex of $L$ for any simplex $\sigma \in K$. This construction endows $L^K$ with the structure of a simplicial complex.



\begin{definition}\cite{FTGCMVV} 
Let $K$ be a simplicial complex. The \emph{Moore path complex} of $K$, denoted by $P K$,
is defined as the full subcomplex of $K^\mathbf{Z}$ generated by all Moore paths
$\gamma \colon \mathbf{Z} \to K$.
\end{definition}

A finite set $\{\gamma_0,\dots,\gamma_p\} \subset PK$ determines a simplex of $PK$ if and
only if, for every integer $i \in \mathbb{Z}$, the set
\[
\{\gamma_0(i),\dots,\gamma_p(i),\gamma_0(i+1),\dots,\gamma_p(i+1)\}
\]
forms a simplex of $K$.

An important feature of $PK$ is that, for any bounded interval $[a,b] \subset \mathbf{Z}$,
the simplicial complex $K^{[a,b]}$ naturally appears as a full subcomplex of $PK$.
Furthermore, given a simplicial map $f \colon K \to L$, the composition $f \circ \gamma$
is a Moore path in $L$ for every Moore path $\gamma$ in $K$. 

\begin{theorem} \cite{FTGCMVV}
Let $K$ be an arbitrary simplicial complex. Then the simplicial map
\[
\pi = (\alpha,\omega) \colon PK \longrightarrow K \times K
\]
is a simplicial finite fibration, where the maps $\alpha$ and $\omega$ are those defined in
Definition~\ref{initialfinal}.
\end{theorem}

For more details on Moore paths and Moore path complex, we refer the redear to \cite{FTGCMVV}.\\

\subsection{ $m$-homotopic distance} 
\begin{definition}\cite{MVMLO}
Let $f,g \colon X \to Y$ be two maps and fix an integer $m \geq 0$. We define $D_m(f,g)=k$
to be the smallest integer $k \geq 0$ for which there exists an open cover
$\{U_0,\dots,U_k\}$ of $X$ satisfying the following condition: for every index $j$ and
any map $h \colon P \to U_j$ from an $m$--dimensional cell complex $P$, the compositions
\[
f \circ h \simeq g \circ h
\]
are homotopic.
\end{definition}

\section{$m$-Contiguity Distance}

\begin{definition}
	Given two simplicial maps $\varphi,\psi:K\rightarrow K'$, the m-contiguity distance denoted by $\SD_m(\varphi,\psi)$, is the least integer $k\geq 0$ such that there exists a covering of $K$ by subcomplexes $K_0,\dots,K_k$ with the property that, for each $K_j$, any simplicial map $\eta:P\rightarrow K_j$ from an $m$-dimesional simplicial complex P, $\varphi\circ\eta$ and $\psi\circ\eta$ are in the same contiguity class.
\end{definition}

\begin{proposition} \label{prop1}
	For simplicial maps $\varphi,\psi:K\rightarrow K'$, we have the following 
	\begin{itemize}
	\item[(1)]  $\SD_m(\varphi,\psi)\leq\SD(\varphi,\psi).$
    \item[(2)]  If $n\leq m$, then $\SD_n(\varphi,\psi)\leq \SD_m(\varphi,\psi).$
    \end{itemize}
\end{proposition}

The following example gives a strict case for Proposition~\ref{prop1} (1).

 \begin{example}
    Consider the abstract simplicial complex $K$ as given in Figure~\ref{f1}. Let us take the identity map $\id: K \to K$ and the constant map $c_0: K \to K$, which sends every simplex of $K$ to the vertex $0$. We know that $\SD(\id,c_0)=\scat(K)=1$ in \cite{BPV}. To compute $0-$contiguity distance, let us choose $P$ as a 0-dimensional complex, $P = \{a\}$. In this case, there are three possible maps $\eta_i: P \to K$, and for each $i$, 
$\varphi \circ \eta_i$ and $\psi \circ \eta_i$ are contiguous, hence they belong to the same contiguity class. 
Therefore, we obtain $SD_0(\mathrm{id}, c_0) = 0$, which shows that $SD_0(\mathrm{id}, c_0) < SD(\mathrm{id}, c_0)$.

    \begin{figure}
\begin{center}
\begin{tikzpicture}[scale=0.5]
    \coordinate[label=below:$v_0$] (0) at (0,0);
    \coordinate[label=above:$v_1$] (1) at (2,3.4);
    \coordinate[label=below:$v_2$] (2) at (4,0);

    \draw (0) -- (1);
    \draw (0) -- (2);
    \draw (2) -- (1);
    
    \foreach \i in {0,1,2} {
        \filldraw[black] (\i) circle (3pt);
    }
\end{tikzpicture}
\end{center}
\caption{} \label{f1}
\end{figure}
\end{example}

\begin{example}
    Let $ K = \partial \Delta^3 $ be the simplicial complex consisting of all proper faces of the $3$-simplex $ \Delta^3 $, with vertex set $\{w_0,w_1,w_2,w_3\}$. Let $P$ be the simplicial complex given by Figure~\ref{f2}, and define a simplicial map $\eta:P\to K$ by $\eta(v_0)=w_0,  \eta(v_1)=w_1 $ and $\eta(v_2)=w_2.$ Let $id:K\to K $ denote the identity map. 
    
    First consider the case where $ \varphi : K \to K $ is the constant simplicial map defined by $ \varphi(w) = w_3 $ for all vertices $ w \in K $.  Then it is obvious that $\id\circ\eta \nsim_c \varphi\circ\eta.$
    
     If $\varphi$ is not constant and its image contains more than one vertex, then $\id\circ\eta$ is again not contiguous to $\varphi\circ\eta$. 
     
     Therefore, $\id\circ\eta$ and $\varphi\circ\eta$ do not belong to the same contiguity class for any simplicial map $\varphi:K\to K$.

    In particular, for the constant map $c_3:K\to K$ defined by $ c_3(w) = w_3 $ for all vertices $ w \in K. $ We obtain $\SD_2(\id,c_3)\neq 0.$

   On the other hand, as shown in \cite{M}, $\scat(K)=1.$ Then, by the Proposition~\ref{prop1}, it follows that $\SD_2(\id,c_3)=1.$
    \begin{figure}

\begin{tikzpicture}[scale=0.5]
    \coordinate[label=below:$v_0$] (0) at (0,0);
    \coordinate[label=above:$v_1$] (1) at (2,3.4);
    \coordinate[label=below:$v_2$] (2) at (4,0);

    \fill[gray!30] (0) -- (1) -- (2) -- cycle;
    
    \draw (0) -- (1);
    \draw (0) -- (2);
    \draw (2) -- (1);
    
    \foreach \i in {0,1,2} {
        \filldraw[black] (\i) circle (3pt);
    }
\end{tikzpicture}

\caption{} \label{f2}
\end{figure} 
\end{example}

\begin{proposition} \label{pro32}
    If $\varphi\sim\varphi', \psi\sim\psi':K\to K',$ then $\SD_m(\varphi,\psi)=\SD_m(\varphi',\psi').$
\end{proposition}

\begin{proposition}\label{yprop3.3}
    Let $\varphi,\psi:K\to K'$ and $\mu:M\to K$ be simplicial maps. Then we have $$\SD_m(\varphi\circ\mu,\psi\circ\mu)\leq\SD_m(\varphi,\psi).$$
\end{proposition}

\begin{proof}
    Let $\SD_m(\varphi,\psi)=k.$ Then there exist subcomplexes $K_0,\dots,K_k$ of K such that each $K_j$ has the property that any map $\eta:P\to K_j$ from an $m$-dimensional simplicial complex P, $\varphi\circ\eta\sim \psi\circ\eta$. Define $M_j:=\mu^{-1}(K_j)$ and take any simplicial map $\eta':P\to M_j.$ Since $\mu\circ\eta'$ is a map from $P$ to $K_j,$ we have $\varphi\circ\mu\circ\eta'\sim\psi\circ\mu\circ\eta'.$ It follows that $\SD_m(\varphi\circ\mu,\psi\circ\mu)\leq k.$
\end{proof}

\begin{corollary}
    Let $\varphi,\psi:K\to K' $ be simplicial maps and $\beta:M\to K$ be a simplicial map which has right strong homotopy inverse (that is, $\beta$ satisfies $\beta\circ\alpha\sim\id_K$ where $\alpha:K\to M$). Then $\SD_m(\varphi\circ\beta,\psi\circ\beta)=\SD_m(\varphi,\psi).$
\end{corollary}

\begin{proof}
    As $\beta\circ\alpha\sim\id_K$, we have $\varphi\circ\beta\circ\alpha\sim\varphi$ and $\psi\circ\beta\circ\alpha\sim\psi.$ Hence, $$\SD_m(\varphi,\psi)=\SD_m(\varphi\circ\beta\circ\alpha,\psi\circ\beta\circ\alpha)\leq \SD_m(\varphi\circ\beta,\psi\circ\beta)\leq\SD_m(\varphi,\psi)$$
    where the equality and inequalities follow from Proposition \ref{yprop3.3} and \ref{pro32}.
\end{proof}

\begin{proposition}\label{yprop3.4}
    Let $\varphi,\psi: K\to K'$ and $\mu:K'\to M$ be simplicial maps. Then we have $$\SD_m(\mu\circ\varphi,\mu\circ\psi)\leq \SD_m(\varphi,\psi).$$
\end{proposition}

\begin{proof}
  Let $\SD_m(\varphi,\psi)=k.$ Then there exist subcomplexes $K_0,\dots,K_k$ of K such that each $K_j$ has the property that any map $\eta:P\to K_j$ from an $m$-dimensional simplicial complex P, $\varphi\circ\eta\sim \psi\circ\eta$. Then, $\mu\circ\varphi\circ\eta\sim\mu\circ\psi\circ\eta.$ Therefore, $\SD_m(\mu\circ\varphi,\mu\circ\psi)\leq k.$
\end{proof}

\begin{corollary}
    Let $\varphi,\psi:K\to K'$ be simplicial maps and $\alpha:K'\to M$ be a simplicial map which has left strong homotopy inverse. Then 
$\SD_m(\alpha\circ\varphi,\alpha\circ\psi)= \SD_m(\varphi,\psi).$
\end{corollary}

\begin{proof}
    Since $\alpha$ has left strong homotopy inverse, we have $\beta\circ\alpha\circ\varphi\sim\varphi$ and $\beta\circ\alpha\circ\psi\sim\psi.$ Thus, $$\SD_m(\varphi,\psi)=\SD_m(\beta\circ\alpha\circ\varphi,\beta\circ\alpha\circ\psi)\leq\SD_m(\alpha\circ\varphi,\alpha\circ\psi)\leq\SD_m(\varphi,\psi)$$ where the equality and inequalities follow from Proposition \ref{yprop3.4} and \ref{pro32}.
\end{proof}

The following theorem shows that the $m$-contiguity distance is a strong homotopy invariant.
\begin{theorem}\label{StrongHI}
    Let $\beta: K'\to K$ and $\alpha:L\to L'$ have right and left strong homotopy inverses, respectively. If the simplicial maps $\varphi,\psi:K\to L$ and $\varphi',\psi':K'\to L'$ make the following diagram commutative up to contiguity, then we have $\SD_m(\varphi,\psi)=\SD_m(\varphi',\psi').$

\[
\xymatrix{
K \ar@<0.5ex>[r]^{\varphi} \ar@<-0.5ex>[r]_{\psi} &
L \ar[d]^{\alpha} \\
K' \ar@<0.5ex>[r]^{\varphi'} \ar@<-0.5ex>[r]_{\psi'} \ar[u]_{\beta} &
L'
}
\]
\end{theorem}

\begin{theorem} \label{dimen}
    For two simplicial maps $\varphi,\psi:K\to K'$ we have $$\SD_{dim(K)}(\varphi,\psi)=\SD(\varphi,\psi).$$
\end{theorem}

\begin{proof}
    Let $m=dim(K).$ Proposition \ref{prop1} implies that $\SD_m(\varphi,\psi)\leq \SD(\varphi,\psi).$ 
    
    Now, suppose that $\SD_m(\varphi,\psi)=k.$ There exist subcomplexes $K_0,\dots,K_k$ of K such that each $K_j$ has the property that any map $\eta:P\to K_j$ from an $m$-dimensional simplicial complex P, $\varphi\circ\eta\sim \psi\circ\eta$. If $\eta$ is chosen to be surjective and to satisfy $K_j\subseteq \operatorname{Im}(\eta)$ for each $j$, then $ \varphi|_{K_j}\sim\psi|_{K_j} $ for every $j=0,\ldots,k$. This completes the proof.
\end{proof}

\begin{corollary}
    For any simplicial maps $\varphi,\psi:K\to K'$ with finite $\dim(K)$, Proposition \ref{prop1}(2) yields $$\SD_0(\varphi,\psi)\leq\ldots\leq \SD_{dim(K)}(\varphi,\psi)= \SD_{dim(K)+1}(\varphi,\psi)=\ldots=\SD(\varphi,\psi)$$
\end{corollary}

\begin{definition}
	Let $K$ be a simplicial complex and fix some integer $m\geq 1$. The $m$- simplicial LS category $scat_m(K)$ is the least integer $k\geq 0$ such that $K$ has a cover $K_0,\dots,K_k$ where $K_0,\dots,K_k$ are subcomplexes of $K$ and each $K_j$ has the property that any simplicial map $\eta:P\rightarrow K_j$ from an $m-$dimensional simplicial complex P, $\iota_j\circ\eta$ and $c_{v_0}\circ\eta$ are in the same contiguity class where $\iota_j:K_j\rightarrow K$ is the inclusion and $c_{v_0}:K_j\rightarrow K$ is the constant map.
\end{definition}

\begin{remark} \label{rem}
	$\scat_m(K)\leq \scat(K)$.
\end{remark}

\begin{lemma} \label{lemma}
	For two simplicial maps $\varphi,\psi:K\rightarrow K'$ we have $$
	\SD_m(\varphi,\psi)\leq \scat_m(K). $$
\end{lemma}

\begin{proof}
	Let $\scat_m(K)=k$. Then there exist subcomplexes $K_0,\dots,K_k$ of K such that each $K_j$ has the property that any map $\eta:P\rightarrow K_j$ from an $m-$ dimensional simplicial complex P,  $\iota_j\circ\eta\sim c_{v_0}\circ\eta$ where   $\iota_j:K_j\rightarrow K$ is the inclusion and $c_{v_0}:K_j\rightarrow K$ is the constant map. 
	$$ \varphi\circ\eta\sim \varphi\circ  c_{v_0} \sim c \sim \psi\circ c_{v_0} $$ is obtained.
\end{proof}

\begin{theorem}\label{sLSinclusion}
	For the inclusions $i_1,i_2:K\rightarrow K^2$ defined by $i_1(\sigma)=(\sigma,v_0)$ and $i_2(\sigma)=(v_0,\sigma)$, we have $$
	\SD_m(i_1,i_2)=\scat_m(K). $$
\end{theorem}

\begin{proof}
	By Lemma \ref{lemma}, we know that $\SD_m(i_1,i_2)\leq\scat_m(K)$. Suppose that $\SD_m(i_1,i_2)=k$. There exist subcomplexes $K_0,\dots,K_k$ of K such that each $K_j$ has the property that any map $\eta:P\rightarrow K_j$ from an $m-$ dimensional simplicial complex P, $i_1\circ\eta\sim i_2\circ\eta$.
	For all $\sigma \in K$,
	$$ i_1\circ\eta(\sigma)=i_1(\eta(\sigma))=(\eta(\sigma),v_0)\sim(v_0,\eta(\sigma)).$$ Namely, $\eta$ and $ c_{v_0}$ are in the same contiguity class.

\end{proof}

\begin{theorem}\label{scat}
    For the identity map $ \id:K\rightarrow K $ and the constant map $c_{v_0}:K\rightarrow K$, we have $$ \scat_m(K)=SD_m(id,c_{v_0}).$$
\end{theorem}
\begin{proof}
    Suppose that $\SD_m(id,c_{v_0})=k.$ Then there exist subcomplexes $K_0,\dots,K_k$ of $K$ such that each $K_j$ has the property that any map $\eta:P\to K_j$ from $m-$ dimensional simplicial complex P, $id\circ\eta\sim c_{v_0}\circ\eta$. Then we say that inclusion from $K_j$ to $K$ composite $\eta$ and resctriction of $c_{v_0}$ on $K_j$ are same contiguity class for all $K_j$, that is $\iota_j\circ\eta\sim c_{v_0}\circ\eta$. We conclude that $\scat_m(K)\leq k$.

   The other direction follows from the Lemma \ref{lemma}.
\end{proof}

\begin{corollary}
    $\scat_{dim(K)}(K)=\scat(K).$
\end{corollary}

\begin{proof}
    Let $m=dim(K)$. $\scat_m(K)\leq \scat(K) $ follows from Remark \ref{rem}. 
    
    For the reverse inequality, taking $id:K\to K$ to be the identity map and $c_{v_0}:K\to K$ to be the constant map in Theorem \ref{dimen} , and using Theorem \ref{scat}, we obtain the required equality.
\end{proof}

\begin{proposition} \label{collapse}
    If $K$ is strongly collapsible then $\scat_m(K)=0.$
\end{proposition}

\begin{proof}
    If $K$ is strongly collapsible then $\scat(K)=0$. From Remark \ref{rem} , we have $\scat_m(K)=0.$
\end{proof}

\begin{proposition} 
    If $K$ is strongly collapsible then $\SD_m(\varphi,\psi)=0$ for any $\varphi,\psi:K\to K'$ simplicial maps.
\end{proposition}

\begin{proof}
    From the Proposition \ref{collapse}, we have $\scat_m(K)=0$. If we use Lemma \ref{lemma}, we conclude that $\SD_m(\varphi,\psi)=0.$
\end{proof}

The converse does not hold in general, as shown by the following example.

\begin{example}
    
Consider the simplicial complex $K$ given in Figure \ref{pic}. Let  $id:K\to K $ be the identity map and, let $c_0:K\to K$ denote the constant map sending every vertex of $K$ to the vertex $0$. Let $P$ be a $0$-dimensional simplicial complex. For a simplicial map $\eta:P\to K$, there are exactly six possibilities, denoted by $\eta_0,\ldots,\eta_5$ each of which is constant. Moreover, 
\[
id\circ\eta_j\sim_c c_0\circ\eta_j
\]
for $j=0,1,2,3,4$, while
\[
id\circ\eta_5\sim_c\eta_3\sim_c c_0\circ\eta_5.
\]
Hence, $id\circ\eta_j$ and $c_0\circ\eta_j$ belong to the same contiguity class for every $j=0,\ldots,5$. Therefore, $\SD_0(id,c_0)=0.$ 

On the other hand, it is known from \cite{FTMVV} that $K$ is not strongly collapsible.

\begin{figure}
\centering
\begin{tikzpicture}[scale=0.5]

\fill[gray!40] (0,0) -- (4,6.8) -- (8,0) -- cycle;

\draw (0,0) -- (3,2.8);
\draw (0,0) -- (4,1);
\draw (4,6.8) -- (3,2.8);
\draw (4,6.8) -- (5,2.8);
\draw (8,0) -- (4,1);
\draw (8,0) -- (5,2.8);
\draw (3,2.8) -- (4,1);
\draw (3,2.8) -- (5,2.8);
\draw (5,2.8) -- (4,1);

\draw (0,0) -- (4,6.8);
\draw (0,0) -- (8,0);
\draw (8,0) -- (4,6.8);

\filldraw (0,0) circle (4pt);
\filldraw (4,6.8) circle (4pt);
\filldraw (8,0) circle (4pt);
\filldraw (3,2.8) circle (4pt);
\filldraw (4,1) circle (4pt);
\filldraw (5,2.8) circle (4pt);

\node[below left]  at (0,0) {$0$};
\node[above]       at (4,6.8) {$1$};
\node[below right] at (8,0) {$2$};

\node[left]        at (3,2.8) {$3$};
\node[below]       at (4,1) {$4$};
\node[right]       at (5,2.8) {$5$};

\end{tikzpicture}
\caption{}
\label{pic}
\end{figure}
\end{example}
In \cite{BPV}, Borat, Pamuk and Vergili proved inequality for geometric relations and barycentric subdivision for simplicial maps in the setting of the contiguity distance. In order to completeness of our study of $m-$ contiguity distance of simplicial maps, we state them and give their proofs in our settings.  

\begin{theorem}\label{geometric}
For two simplicial maps $\varphi,\psi:K\rightarrow K'$ we have 
    $$ D_m(|\varphi|,|\psi|)\leq SD_m(\varphi,\psi). $$
\end{theorem}

\begin{proof}
    Suppose $\SD_m(\varphi,\psi)=k.$ Then there exist a covering of $K$ by subcomplexes $K_0,\dots,K_k$ with the property that, for each $K_j$ any simplicial map $\eta:P\to K_j$ from an $m-$ dimensional simplicial complex $P$, $\varphi\circ\eta$ and $\psi\circ\eta$ are in the same contiguity class. The union of the closed subsets $|K_0|,\dots,|K_k|$ covers $|K|$ and the geometric realisations of $\varphi\circ\eta$ and $\psi\circ\eta$ are homotopic continuous maps. Since the geometric realisation is a covariant functor, we have $|\varphi|\circ|\eta|\simeq |\psi|\circ|\eta|.$
\end{proof}

\begin{definition}
The barycentric subdivison of a given simplicial complex $K$ is the simplicial complex $\sd(K)$ whose set of vertices is $K$ and each n-simplex in $\sd(K)$ is of the form $\{\sigma_{0},\sigma_{1},\cdots,\sigma_{n}\}$ where $\sigma_{0}\subsetneq \sigma_{1}\subsetneq,\cdots,\subsetneq\sigma_{n}$. 
\end{definition}

 \begin{definition}
     For a simplicial map $\phi:K\to L$, the induced map $\sd(\phi):\sd(K)\to\sd(L)$ is given by $\sd(\phi)(\{\sigma_{0},\sigma_{1},\cdots,\sigma_{n}\})=\{\phi(\sigma_{0}),\phi(\sigma_{1}),\cdots,\phi(\sigma_{n})\}$
 \end{definition}
Notice that $\sd(\phi)$ is a simplicial map, $\sd(\id)=\id$ and $\sd(\phi\circ\psi)=\sd(\phi)\circ\sd(\psi)$.

\begin{proposition}[\cite{FTMVMV2}]\label{barycentric}
If the simplicial maps $\phi,\psi:K\to L$ are in the same contiguity class, so are $\sd(\phi)$ and $\sd(\psi).$
\end{proposition} 

\begin{proposition}\label{m-barycentric}
Let $\phi,\psi:K\to L$ be the simplicial maps. If given the simplicial map $\eta:P\to K$ with $\phi\circ\eta$ and $\psi\circ\eta$ are in the same contiguity class, then $\sd(\phi)\circ\sd(\eta)$ and $\sd(\psi)\circ\sd(\eta)$ are in the same contiguity class.
\end{proposition} 

\begin{proof}
    By Proposition \ref{barycentric}, we obtain that $\sd(\phi\circ\eta)$ and $\sd(\psi\circ\eta)$ are in the same contiguity class since $\phi\circ\eta$ and $\psi\circ\eta$ are in the same contiguity class. On the other hand, the barycentric subdivision preserves the compositions of maps, i.e. $\sd(\phi\circ\eta)=\sd(\phi)\circ\sd(\eta)$ and $\sd(\psi\circ\eta)=\sd(\psi)\circ\sd(\eta)$. This completes the proof.
\end{proof}

We begin with the following lemma, which provides the key technical ingredient in the proof of Theorem~\ref{SD-barycentric}. We include its proof for the convenience of the reader and to keep the exposition self-contained.

\begin{lemma}\label{simp-approx}
Let $f:K\to L$ be a simplicial map. Then $f\circ c_K:\sd(K)\longrightarrow L$ is contiguous to $c_L\circ\sd(f):\sd(K)\longrightarrow L,$ where
$c_K:\sd(K)\to K,$ $c_L:\sd(L)\to L$ are the canonical subdivision maps.
\end{lemma}

\begin{proof}
Let $v_\sigma$ be a vertex of $\sd(K)$ corresponding to a simplex
$\sigma$ of $K$. The vertex $(f\circ c_K)(v_\sigma)$ is one of the vertices of the simplex $f(\sigma)$. On the other hand,
$(c_L\circ\sd(f))(v_\sigma)=c_L(v_{f(\sigma)})$ is also a vertex of the simplex $f(\sigma)$.

Now let $\tau=[v_{\sigma_0},\ldots,v_{\sigma_r}]$ be a simplex of $\sd(K)$, where $\sigma_0\subset\cdots\subset\sigma_r.$ Since $f$ is simplicial,
$f(\sigma_i)\subset f(\sigma_r)$ for every $i$.

Therefore every vertex appearing in
$(f\circ c_K)(\tau)$ or $(c_L\circ\sd(f))(\tau)$ lies in the simplex $f(\sigma_r)$ of $L$.
Hence $(f\circ c_K)(\tau)\cup(c_L\circ\sd(f))(\tau)$ is a simplex of $L$, proving that $f\circ c_K\sim_c c_L\circ\sd(f).$
\end{proof}

The relation between the contiguity distance of two simplicial maps and the contiguity distance of their induced maps on barycentric subdivisions can be given as follows.

\begin{theorem}\label{SD-barycentric}
    For simplicial maps $\phi,\psi:K\to L$ , then $$SD_{m}(\sd(\phi),\sd(\psi))\leq SD_{m}(\phi,\psi)$$
\end{theorem} 

\begin{proof}
Let $\SD_m(\varphi,\psi)=k .$ Then there exist subcomplexes $K_0,\ldots,K_k$ of $K$ such that
$K=\bigcup_{\ell=0}^{k}K_\ell$ and, for each $\ell\in\{0,\ldots,k\}$, every $m$-dimensional simplicial complex $P$, and every simplicial map $\nu:P\longrightarrow K_\ell,$ the maps $\varphi\circ\nu,\psi\circ\nu$ belong to the same contiguity class.

Consider the cover $\sd(K_0),\ldots,\sd(K_k)$ of $\sd(K)$. We claim that this cover satisfies the defining property of $\SD_m(\sd(\varphi),\sd(\psi)).$ 

Let us fix $K_l$. Let $P$ be any $m$-dimensional simplicial complex, and $\eta:P\longrightarrow\sd(K_\ell)$ be any simplicial map.

Since $c\circ\eta:P\to K_\ell$ is a simplicial map where $c:\sd(K_\ell)\to K_\ell$ is the canonical barycentric subdivision map, 
$\varphi\circ c\circ\eta,\psi\circ c\circ\eta$ belong to the same contiguity class.

By Proposition~\ref{m-barycentric},
$\sd(\varphi)\circ \sd(c)\circ\sd(\eta),\sd(\psi)\circ \sd(c)\circ\sd(\eta)$  are in the same contiguity class.

On the other hand, since $\sd(c)\circ \sd(\eta)=c_{\sd K_l}\circ \sd(\eta)$, they are also in the same contiguity class. Thus, $\sd(\varphi)\circ c_{\sd K_l}\circ \sd(\eta)\sim\sd(\psi)\circ c_{\sd K_l}\circ \sd(\eta)$. 

If we apply Lemma~\ref{simp-approx} to the simplicial map $\eta:P\to\sd(K_\ell)$, then it  follows that
$\eta\circ c_P\sim_c c_{\sd(K_\ell)}\circ\sd(\eta).$ Hence 
$\sd(\varphi)\circ \eta\circ c_P\sim\sd(\psi)\circ \eta\circ c_P$. 

Finally, since $c_P:\sd(P)\to P$ is a simplicial approximation to the identity,
the simplicial approximation theorem implies that
$\sd(\varphi)\circ\eta$ and $\sd(\varphi)\circ\eta\circ c_P$ are contiguous, and similarly $\sd(\psi)\circ\eta$ and $\sd(\psi)\circ\eta\circ c_P$ are contiguous. Hence, $\sd(\varphi)\circ\eta$ and $\sd(\psi)\circ\eta$ belong to the same contiguity class.

Since $\sd(K_\ell)$ is arbitary, we obtain that
$\SD_m(\sd(\varphi),\sd(\psi))\le k.$ This completes the proof of the Theorem.
\end{proof}

\section{Categorical product}

Since the Cartesian product of two simplical complexes may not be a simplicial complex, we used the categorical product instead of the Cartesian product. Let us denote $K\times L$ for the categorical product of the simplicial complexes K and L. It is different from the usual standard notion used in \cite{K} as $K \sqcap L$.  

\begin{definition}
    The categorical product of the simplicial complexes K and L, denoted as $K\times L$, is defined as follows. The vertices of $K\times L$ are pairs $(v,w)$ of vertices with $v\in K$ and $w\in L$, and the simplices of $K\times L$ are the set of vertices $\{(v_{1},w_{1}),\cdots,(v_{s},w_{s})\}$ such that $\{(v_{0},\cdots,v_{s})\}$ is a simplex of K and $\{(w_{0},\cdots,w_{s})\}$ is a simplex of L.  
\end{definition}


The categorical product of simplicial maps $f:K\rightarrow L$ and $g:K'\rightarrow L'$ is defined by $f\times g: K\times L\rightarrow K'\times L'$, $(f\times g)(\sigma,\tau):=(f(\sigma),g(\tau))$.

We state the following easy observations as lemma below since we use them in the proof of product formula. 

\begin{lemma}\label{easy lemma}
 For given simplicial maps $\phi,\psi:K\to L$ and $\phi',\psi':K'\to L'$ with $\phi\sim \psi$ and $\phi'\sim  \psi'$, then $$\phi\times\phi'\sim \psi\times \psi':K\times K'\to L\times L'$$
\end{lemma}

\begin{theorem}\label{categoricalproduct}
    For given simplicial maps $\phi,\psi:K\to L$ and $\phi',\psi':K'\to L'$, $$ SD_m(\phi\times \phi', \psi\times \psi')+1 \leq (SD_m(\phi,\psi)+1)(SD_m(\phi',\psi')+1).$$
\end{theorem}

\begin{proof}
Suppose $SD_{m}(\phi,\psi)=k$ and $SD_{m}(\phi',\psi')=l$. Then there exists open simplicial complexes $U_{0},\cdots,U_{k}$ for which $\eta:P\to U_{i}$ from m-dimenstional simplicial complex P, $\phi\circ \eta$ and $\psi\circ \eta$ are in the same contiguity class.  Similarly, there exists open simplicial complexes $V_{0}\cdots,V_{l}$ for which $\eta:P\to V_{j}$ from m-dimensional simplicial complex P, $\phi'\circ \eta$ and $\psi'\circ \eta$ are in the same contiguity class.

Consider the following collections $\{U_{i}\times V_{j}\}$ for $0\leq i\leq k$, $0\leq j\leq l$. By construction, $K\times K'$ are covered by above collection. 

Now let $P$ be a m-dimensional simplicial complex with simplicial map $h:P\to U_{i}\times V_{j}$. By the definition of categorical product, there exists projections, which are simplicial maps $h_{U_{i}}:P\to U_{i}$ and $h_{V_{j}}:P\to V_{j}$. By $SD_{m}$-property for $U_{i}$ and $V_{j}$, we have the followings: $$\phi\circ h_{U_{i}}\sim \psi\circ h_{U_{i}}$$ and $$\phi'\circ h_{V_{j}}\sim \psi'\circ h_{V_{j}}.$$ On the other hand, the categorical product satisfies composition rule, we have the followings: $$(\phi\times \phi')\circ h=(\phi\circ h_{U_{i}})\times (\phi'\circ h_{V_{j}}),$$ $$(\psi\times \psi')\circ h=(\psi\circ h_{U_{i}})\times (\psi'\circ h_{V_{j}})$$    

By Lemma \ref{easy lemma}, we conclude $(\phi\times \phi')\circ h$ and $(\psi\times \psi')\circ h$ are in the same contiguity class. Therefore, the collections  
$\{U_{i}\times V_{j}\}$ for $0\leq i\leq k$, $0\leq j\leq l$ have $SD_{m}$-property. This finishes the proof of the Theorem.  
\end{proof}

\section{$m$-sectional category (Svarc genus)}

    The following ideas of \cite{FTGCMVV}, we can define "the $m$-simplicial Svarc genus of a simplicial map" and "$m$-homotopy simplicial Svarc genus of a simplicial map" for given a simplicial map $\phi:K\to L$. We prove that two notions are the same when $\phi$ is a simplicial fibration. Naturally, it preserves all known properties simpliar to contiguity distance.    

\begin{definition} \cite{FTGCMVV}
    A simplicial map $p:E\rightarrow B$ is a simplicial fibration if for any simplicial complex $K$, for any simplicial maps $\varphi,\psi: K\to B$ are in the same contiguity class with $\ell$-steps, ie,
    \[
    \varphi=\phi_0\sim_c \ldots\sim_c\phi_{\ell}=\psi
    \]
    
    and for any map $\tilde{\varphi}:K\to E$ such that $p\circ \tilde{\varphi}=\varphi$, there exits a simplicial map $\tilde{\psi}:K\to E$ such that $\tilde{\varphi}$ and $\tilde{\psi}$ are in the same contiguity class with $\ell$-steps, ie,
    \[
    \tilde{\varphi}=\tilde{\phi_0}\sim_c \ldots\sim_c\tilde{\phi}_{\ell}=\tilde{\psi},
    \]
where $p\circ \tilde{\phi_i}=\phi_i$, $0\leq i\leq \ell$.
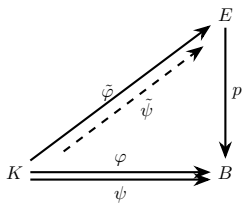
\begin{center}

\begin{tikzpicture}[>=Stealth, thick, scale=0.7, transform shape]

\node (K) at (0,0) {$K$};
\node (B) at (4,0) {$B$};
\node (E) at (4,3) {$E$};

\draw[->] (K) -- node[above] {$\varphi$} (B);
\draw[->] ([yshift=-4pt]K.east) -- ([yshift=-4pt]B.west)
node[midway, below] {$\psi$};

\draw[->] (E) -- node[right] {$p$} (B);

\draw[->] (K) -- node[left] {$\tilde\varphi$} (E);
\draw[->, dashed] ([xshift=18pt,yshift=4pt]K.north east)
-- ([xshift=-4pt,yshift=-10pt]E.south west)
node[midway, below right, yshift=4pt] {$\tilde\psi$};

\end{tikzpicture}
\end{center}
\end{definition}

Let us recall that $l\geq1,$ $I_{l}$ is the one-dimensional simplicial complex whose vertices are the integers $\{0,\cdots,l\}$ and the edges are the pairs $\{j,j+1\}$, for $0\leq j< l$. We sometimes denote $[0,n_{j}]$ instead of $I_{n_{j}}$ if there are more than two subindexes.  

The following is an important characterisation of simplicial fibration. 

\begin{theorem}[\cite{FTGCMVV}, Proposition 3]\label{fibration} A simplical map $p:E\to B$ is a simplicial fibration if and only if given simplicial maps $H:K\times I_{l}\to B$ and $\phi:K\times \{0\}\to E$ as in the following commuative diagram: 

\[
\begin{tikzcd}
K \times \{0\} \arrow[d, hookrightarrow, "\iota^{l}_{0}"] \arrow[r, "\phi"]  & E \arrow[d, "p"] \\
K \times I_{l} \arrow[r,"H"] \arrow[dashed]{ur}{\tilde{H}}  & B
\end{tikzcd}
\]
there exists a simplicial map $\tilde{H}:K\times I_{l}\to E$ such that $\tilde{H}\circ \iota^{l}_{0}=\phi$ and $p\circ \tilde{H}=H$.

\end{theorem}

\begin{definition}
The $m$-simplicial Svarc genus of a simplicial map $\phi:K\to L$, $\secat_{m}(\phi)$, is the the minimum nonnegative integer n such that L is the covered by n+1 subcomplexes $\{L_{0},
\cdots,L_{n}\}$, and each $L_j$ has the property that any simplicial map $\eta:P\rightarrow L_j$ from an $m-$dimensional simplicial complex P, there exists a simplicial map $s:P\to K$ such that $\phi\circ s$ is the the following composition $\iota_j\circ\eta$, where $\iota_j:L_{j}\subset L$ is the inclusion. 
\end{definition}

\begin{definition}
The $m$-homotopy simplicial Svarc genus of a simplicial map $\phi:K\to L$, $\hsecat_{m}(\phi)$, is the minimum nonnegative integer n such that L is covered by n+1 subcomplexes $\{L_{0},
\cdots,L_{n}\}$, and each $L_j$ has the property that any simplicial map $\eta:P\rightarrow L_j$ from a simplicial complex $m$-dimensional $P$, there exists a simplicial map $s:P\to K$ such that $\phi\circ s$ and $\iota_j\circ\eta$ are in the same contiguity class.  
\end{definition}

\begin{remark}\label{homotopy}
    By above definitions, we have the following inequality $\hsecat_{m}(\phi)\leq \secat_{m}(\phi)$. 
\end{remark}

\begin{theorem}\label{hsecat=secat}
    Let $p:E\to B$ be a simplicial fibration. Then $$\hsecat_{m}(p)=\secat_{m}(p).$$ 
\end{theorem}

\begin{proof}
 Due to the Remark \ref{homotopy}, it suffices to show $\secat_{m}(p)\leq \hsecat_{m}(p)$. Let $\hsecat_{m}(p)=k$ and $\{L_{0},\cdots,L_{k}\}$ be covering B with the following property: each $L_j$ has the property that, for any simplicial map $\eta:P\rightarrow L_j$ of a simplicial complex $m$-dimensional P, there exists a simplicial map $s_{j}:P\to E$ such that $p\circ s_{j}$ and $\iota_j\circ\eta$ are in the same contiguity class. By Theorem \ref{fibration} for each j, we have simplicial maps $H_{j}:P\times [0,n_{j}]\to B$ with $H_{j}(-,0)=p\circ s_{j}$ and $H_{j}(-,n_{j})=\iota_j\circ\eta$. 

 Take a lift $\tilde{H_{j}}:P\times [0,n_{j}]\to E$ in the following diagram 

\[
\begin{tikzcd}
P \times \{0\} \arrow[d, hookrightarrow, "\iota"] \arrow[rr, "s_{j}"] & & E \arrow[d, "p"] \\
P \times [0,n_j] \arrow[r,"\eta"] \arrow[dashed]{urr}{\tilde{H_{j}}} \arrow[bend right=30,swap]{rr}{H_{j}} & L_j \arrow[r, "\iota_{j}"] & B
\end{tikzcd}
\]

in such a way that \( p \circ \widetilde{H}_j = H_j \) and
\( \widetilde{H}_j(v,0) = s_{j}(v) \).
Then the map
\[
\xi_j \colon P \to E, \qquad \xi_j(v) := \widetilde{H}_j(v,n_j),
\]
is simplicial and satisfies
\[
p \circ \xi_j(v)
= p \circ \widetilde{H}_j(v,n_j)
= H_j(v,n_j)
= i_j \circ \eta(v).
\]

Thus, $\secat_{m}(p)\leq n$.
\end{proof}

Let us recall the definition of pull-back diagrams 

\[
\begin{tikzcd}
P_{(\phi,\psi)} \arrow[r, "\overline{(\phi.\psi)}"] \arrow[d, "\pi'"] 
& PL \arrow[d, "\pi=(\alpha.w)"] \\
K \arrow[r, "(\phi.\psi)"] & L \times L
\end{tikzcd}
\]
where PL is the set of Moore paths (more details [\cite{FTGCMVV}, section 5.1 Moore paths]), and  $$P_{(\phi,\psi)}=K\times_{L\times L}PL=\{(v,\gamma)\in K\times PL| \phi(v)=\alpha(\gamma), \psi(v)=w(\gamma)\}$$

\begin{theorem}\label{pullback}
Let $\phi,\psi \colon K \to L$ be simplicial maps with the pullback fibration $\pi':P_{(\phi,\psi)}\to K$.  Then
\[
SD_m(\phi,\psi) = \secat_{m}(\pi').
\]
\end{theorem}

\begin{proof} Suppose $\secat_m(\pi')=k$. There exists $k+1$ subcomplexes $K_{0},\cdots,K_{k}$ of $K$ which cover K and for each $j$, $K_j$ has the following property: for any simplicial map $\eta_j:P\to K_{j}$ where $P$ is an $m$-dimensional complex, there exists a simplicial map $s:P\to P_{(\phi,\psi)}$ such that $\pi'\circ s=\iota_j\circ \eta_j$.

\[
\begin{tikzcd}
&& P_{(\phi,\psi)} \arrow[r, "\overline{(\phi,\psi)}"] \arrow[d, "{\pi'}"] 
& PL \arrow[d, "{\pi=(\alpha,\omega)}"] \\
P \arrow[r, swap, "{\eta_j}"] \arrow[urr, "{s}"] 
& K_j \arrow[r, hook, swap, "{\iota_j}"] 
& K \arrow[r, swap, "{(\phi,\psi)}"] 
& L \times L
\end{tikzcd}
\]

By the above commutative diagram, we have
\[
\pi\circ\overline{(\phi,\psi)}\circ s=(\phi,\psi)\circ \pi'\circ s=(\phi,\psi)\circ \iota_j \circ\eta_j=(\phi\circ \iota_j \circ\eta_j,\psi\circ \iota_j \circ\eta_j)=(\phi|_{K_{j}}\circ \eta_j, \psi|_{K_{j}}\circ \eta_j).
\]

Here, $\pi\circ\overline{(\phi,\psi)}$ is a Moore path of length $n_j$ for some $n_j\in \mathbb{Z}$. This allows us to define the following simplicial map

\[
H:P\times [0,n_j]\to L
\]
\[
H(v,i)=\big((\pi\circ \overline{(\phi,\psi)}\circ s)(v)\big)(i)
\]

We have $H(v,0)=\phi|_{K_{j}}\circ \eta_j$ and $H(v,n_j)=\psi|_{K_{j}}\circ \eta_j$. Thus, we have shown that  $\phi|_{K_{j}}\circ \eta$ and $\psi|_{K_{j}}\circ \eta$ are in the same contiguity class with $n_{j}$ steps. Hence, $\SD_{m}(\phi,\psi)\leq \secat_{m}(\pi')$.




Conversely, suppose $SD_{m}(\phi,\psi)=k.$ There exists $k+1$ subcomplexes of K, namely $\{K_{0},\cdots,K_{k}\}$, and each j, we have $\phi|_{K_{j}}\circ \eta_j$ and $\psi|_{K_{j}}\circ \eta_j$ are in the same contiguity class with $n_{j}$ steps, where $\eta_j:P\to K_{j}$ is a simplicial map from an $m$-dimensional simplicial complex $P$. Then we have the family of simplicial maps $$H_{j}:P\times [0,n_{j}]\to L$$ such that $H_{j}(-,0)=\phi|_{K_{j}}\circ \eta$ and $H_j(-,n_{j})=\psi|_{K_{j}}\circ \eta$. Define $G_i: P\to L$ by $G_{i}(-)(t)=H_{i}(-,t)$ where $0\leq i\leq n_j$, $i\in \mathbb{Z}$.

\[
\begin{tikzcd}
&& P_{(\phi,\psi)} \arrow[r, "\overline{(\phi,\psi)}"] \arrow[d, "{\pi'}"] 
& PL \arrow[d, "{\pi=(\alpha,\omega)}"] \\
P \arrow[r, swap, "{\eta_j}"] \arrow[urr, "{s}",dashed] \arrow[bend left=60,swap]{urrr}{\tilde{G_{i}}}
& K_j \arrow[r, hook, swap, "{\iota_j}"] 
& K \arrow[r, swap, "{(\phi,\psi)}"] 
& L \times L
\end{tikzcd}
\]

Since $\pi:PL\to L\times L$ is a simplicial fibration, the pull-back simplicial map $\pi:P_{(\phi,\psi)}\to K$ is also a simplicial fibration. 
$$\pi\circ \tilde{G}_{i}=(\phi,\psi)\circ\iota_{j}\circ\eta_{j}$$

By the universality of the pullback diagram, there exists a simplicial map $s_{j}:P\to P_{(\phi,\psi)}$ such that $\phi\circ s$ and $\iota_j\circ\eta_{j}$ in the same contiguity class.  Therefore, $\secat_{m}(\pi')\leq SD_{m}(\phi,\psi)$. This completes the proof of the Theorem.



\end{proof}

Before we give the application of Theorem \ref{pullback}, we introduce the general notion of $m$--simplicial LS category of simplical map $\phi:K\to L$ below. 

\begin{definition}
	Let $\phi:K\to L$ be simplicial map between simplicial complexes and fix some integer $m\geq 1$. The $m$--simplicial LS category $scat_m(\phi)$ is the least integer $k\geq 0$ such that $K$ has a cover $K_0,\dots,K_k$ where $K_0,\dots,K_k$ are subcomplexes of $K$ and each $K_j$ has the property that any simplicial map $\eta:P\rightarrow K_j$ from an $m$-dimensional simplicial complex P, $\phi\circ\iota_j\circ\eta$ and $c_{v_0}\circ\iota_j\circ\eta$ are in the same contiguity class where $\iota_j:K_j\rightarrow K$ is the inclusion and $c_{v_0}:K\rightarrow L$ is the constant map, $v_{0}\in L$. 
\end{definition}

Note that $\phi$ is the identity simplicial map $\id_{K}$, we recover the above $m-$ simplicial LS category of simplicial complex $K$, i.e $\scat_{m}(\id_{K})=\scat_{m}(K)$. 

\begin{corollary}
    Let $\phi:K\to L$ be a simplicial map and $c_{v_{0}}$ be a constant map. Then $SD_{m}(\phi,c_{v_{0}})=\scat_{m}(\phi)$. 
\end{corollary}

\begin{proof}
    By Theorem \ref{pullback}, it suffices to show that the pullback space $P_{(\phi,c_{v_{0}})}$ is $P_{0}L$ which consists of all Moore paths $\gamma\in PL$ whose initial point is the base vertex $v_{0}\in L$. However, it is clear from the definition of pullback diagram, i.e $P_{0}L=P_{(c_{v_{0}},\phi)}$  
\[
\begin{tikzcd}
P_{(c_{v_{0}},\phi)} \arrow[r, "\overline{(c_{v_{0}}.\phi)}"] \arrow[d, "\pi'"] 
& PL \arrow[d, "\pi=(\alpha.w)"] \\
K \arrow[r, "(c_{v_{0}}.\phi)"] & L \times L
\end{tikzcd}
\]
    
\end{proof}

\begin{remark}
    Let $\id_{K}:K\to K$ be a identity simplical map and $c_{v_{0}}$ be a constant map. Then $$SD_{m}(\id_{K},c_{v_{0}})=\scat_{m}(K).$$ This is another proof of the Theorem \ref{scat}.  
\end{remark}

\section{$m$-Discrete Topological complexity}

\begin{definition}
The $m$-discrete topological complexity $\TC^m(K)$ of a simplicial complex $K$ is the least non-negative integer $k$ such that $K^2$ can be covered by subcomplexes $\Omega_0, \cdots,\Omega_k$ of $K^2$, each of which satisfies that there exists simplicial map $\sigma_j : \Omega_j \to K$ such that for any map $\eta_j:P\rightarrow \Omega_j$  from an $m$-dimensional simplicial complex $P$, $\Delta\circ \sigma_j\circ \eta_j \sim \iota_{j}\circ \eta_j$ holds where $\iota_j : \Omega_j \hookrightarrow K^2$ is the inclusion and $\Delta:K\rightarrow K^2$ is the diagonal map.
\end{definition}

Such $\Omega\subset K^2$ subcomplexes are called \textit{Farber subcomplex of dimension $m$} or simply \textit{Farber subcomplex} if it is clear from the context. Notice that it should not be confused with the $m$-Farber subcomplex as in \cite{ABCD, FTMVMV1} which refers to the Farber subcomplexes used to set up the $m$-th discrete topological complexity. 

\begin{theorem}\label{TC^{m}}
  $\TC^m(K)=\SD_m(pr_1,pr_2).$  
\end{theorem}

\begin{proof}
We first show that $\SD_m(pr_1,pr_2)\leq TC^{m}(K)$. Suppose $TC^{m}(K)=k$. Then there is a covering for $K^{2}$
which consists of Farber subcomplexes $\{\Omega_{0},\cdots,\Omega_{k}\}$ of $K^{2}$, each of which satisfies that there exists simplicial map $\sigma_j : \Omega_j \to K$ such that for any map $\eta_j:P\rightarrow \Omega_j$  from an $m$-dimensional simplicial complex $P$, $\Delta\circ \sigma_j\circ \eta_j \sim \iota_{j}\circ \eta_j$ holds where $\iota_j : \Omega_j \hookrightarrow K^2$ is the inclusion and $\Delta:K\rightarrow K^2$ is the diagonal map.

For each $\Omega_{j}$, we have the followings:

$$\Delta\circ \sigma_j\circ \eta_j \sim \iota_{j}\circ \eta_j$$

$$pr_{1}\circ(\Delta\circ \sigma_j\circ \eta_j) \sim \pr_{1}\circ\iota_{j}\circ \eta_j=pr_{1}|_{\Omega_{j}}\circ \eta_{j}$$

$$pr_{2}\circ(\Delta\circ \sigma_j\circ \eta_j) \sim \pr_{2}\circ\iota_{j}\circ \eta_j=pr_{2}|_{\Omega_{j}}\circ \eta_{j}$$

Since $pr_{1}\circ(\Delta\circ \sigma_j\circ \eta_j)=pr_{2}\circ(\Delta\circ \sigma_j\circ \eta_j)$, we have $pr_{1}|_{\Omega_{j}}\circ \eta_{j}\sim pr_{2}|_{\Omega_{j}}\circ \eta_{j}$. This gives $\SD_{m}(pr_{1},pr_{2})\leq k.$

Conversely, $\SD_{m}(pr_{1},pr_{2})=k$. Then there exist subcomplexes $\{\Omega_{0},\cdots,\Omega_{k}\}$ which their union cover $K\times K$ and $pr_{1}|_{\Omega_{j}}\circ\eta\sim pr_{2}|_{\Omega_{j}}\circ\eta$  with $\eta:P\to \Omega_{j}$. By the definition of the contiguity class, there exists a finite sequence of simplicial maps $\phi_{i}^{j}:\Omega_{j}^{2}\to\Omega_{j}$ such that $$pr_{1}|_{\Omega_{j}}\circ\eta=\phi_{1}^{j}\sim_{c}\phi_{2}^{j}\sim_{c}\cdots\sim_{c}\phi_{l}^{j}= pr_{2}|_{\Omega_{j}}\circ\eta.$$ To be precise, there exists an element $([x],[y])\in \Omega_{j}$ where $[x]=\{x_{0},\dots,x_{l}\}$ and $[y]=\{y_{0},\dots,y_{l}\}$, 
$$\phi_{1}^{j}([x],[y])\cup\phi_{l}^{j}([x],[y])=\{x_{0},\dots,x_{l},y_{0},\dots,y_{l}\}$$
is a simplex in $K$.

We define a simplicial map $\sigma_{j}:\Omega_{j}\to K$, so that 
\[
\begin{tikzcd}
P\arrow[r,"\eta"] & \Omega_{j}\arrow[r,"\sigma_{j}"] & K \arrow[r,"\Delta"] & K^{2}
\end{tikzcd} 
\]
 $\Delta\circ \sigma_j\circ \eta_j \sim \iota_{j}\circ \eta_j$.

Define $$\sigma_{j}([x],[y])=\phi_{1}^{j}([x],[y])\cup\phi_{l}^{j}([x],[y])=\{x_{0},\dots,x_{l},y_{0},\dots,y_{l}\}.$$

$$\Delta\circ\sigma_{j}\circ\eta([x],[y])=(\{x_{0},\dots,x_{l},y_{0},\dots,y_{l}\},\{x_{0},\dots,x_{l},y_{0},\dots,y_{l}\})$$.

$$(\iota_{j}\circ\eta)([x],[y])=([x],[y])=(\{x_{0},\dots,x_{l}\},\{y_{0},\dots,y_{l}\}).$$ Thus, $\Omega_{j}$ is a Farber subcomplex. Since $0\leq j\leq k$, we get $TC(K)\leq k$. This completes the proof of the theorem.  

\end{proof}

\begin{corollary}\label{tc}
    $\TC^m(K)\leq \TC(K)$.
\end{corollary}

\begin{proof} The proof follows directly from the definitions of the usual and $m$-discrete topological complexities as well as from Proposition~\ref{prop1} and Theorem~\ref{TC^{m}}.
     
\end{proof}

\begin{corollary}
    $\TC^{dim(K)}(K)=\TC(K).$
\end{corollary}

\begin{proof}
    Let $m=dim(K).$ By Corollary \ref{tc} we have $\TC^m(K)\leq\TC(K)$. 
  The reverse inequality follows from Theorem \ref{dimen} and Theorem \ref{TC^{m}}.
\end{proof}

\begin{theorem}
    For a simplicial complex $K$, we have $$\scat_{m}(K)\leq TC^{m}(K)\leq \scat_{m}(K^{2})$$
\end{theorem}

\begin{proof}
By Theorem \ref{scat}, $\scat_{m}(K)=SD_{m}(\id_{K},c_{v_{0}})$ and similarly Theorem \ref{TC^{m}}, $\SD_{m}(pr_{1},pr_{2})=TC^{m}(K)$.

The first inequality follows Proposition 1.1 and the following observations $pr_{1}\circ\iota_{1}=\id_{K}$ and $pr_{2}\circ\iota_{1}=c_{v_{0}}$.

$$\scat_{m}(K)=SD_{m}(\id_{K},c_{v_{0}})=\SD_{m}(pr_{1}\circ\iota_{1},pr_{2}\circ\iota_{1})\leq \SD_{m}(pr_{1},pr_{2})=TC^{m}(K)$$

The second inequality follows from Lemma \ref{lemma} and Theorem \ref{TC^{m}} since $pr_{1},pr_{2}:K^{2}\to K$ projection of $K^{2}$ first and second coordinates respectively.
    $$TC^{m}(K)=\SD_{m}(pr_{1},pr_{2})\leq \scat(K^{2})$$
\end{proof}

We would like to finish our paper $m$-topological complexity can be defined in terms of Svarc genus ($m$-sectional category).

\begin{theorem}\label{hsecat}
 For given simplicial complex $K$, $TC^{m}(K)=\hsecat_{m}(\Delta)$, where $\Delta:K\to K^{2}$ is the diagonal map. 
\end{theorem}

\begin{proof}
    Suppose $\secat_m(\Delta)=k$. Then there exists $K_0, \ldots, K_k$ subcomplexes of $K^2$ covering $K^2$ such that for each $j$, for any simplicial map $\eta_j:P\to K_j$ where $P$ is an $m$-dimensional simplicial complex, there exists $s:P\to K$ satisfying $\Delta\circ s\sim \iota_j\circ \eta_j$.

    Define $\sigma_j: K_j\to K$ by $\sigma_j=\pr_1\circ \iota_j$ where $\pr_1: K^2\to K$ is the projection to the first factor. For each $j$, $\sigma_j$ satisfies that for any map $\eta_j:P\to K_j$, we have

    \[
\begin{aligned}
\Delta\circ \sigma_j \circ \eta_j 
    &= \Delta \circ (\pr_1 \circ \iota_j) \circ \eta_j \\
    &= (\Delta\circ \pr_1) \circ (\iota_j\circ \eta_j) \\
    &\sim (\Delta\circ \pr_1)\circ (\Delta\circ s) \\
    &= \Delta \circ s \\
    &\sim \iota_j \circ \eta_j.
\end{aligned}
\]

Hence it follows $\TC^{m}(K)\leq k$.

Conversely, suppose that $\TC^{m}(K)= k$. That is, there exists $\Omega_0,\ldots,\Omega_k$ subcomplexes of $K^2$ covering $K^2$ such that for each $j$, there exists $\sigma_j: \Omega_j\to K$ satisfying that for any map $\eta_j: P \to \Omega_j$, where $P$ is $m$-dimensional simplicial complex, $\Delta\circ \sigma_j\circ \eta_j \sim \iota_j\circ \eta_j$ holds. So if we define $s: P\to K$ by $s=\sigma_j\circ \eta_j$, then it follows that $\secat_m(\Delta)\leq k$. 
    
    

    
\end{proof}

Theorem 4 in \cite{FTGCMVV} can be restated for finite simplicial complexes with the help of Corollary 2 in \cite{FTGCMVV} as follows.

\begin{theorem}[\cite{FTGCMVV}, Theorem 4]\label{Thm4}
    Let $\phi: K\to L$ be a simplicial map for finite simplicial complexes $K$ and $L$. Then $\phi=p\circ \beta$ where $\beta$ is a strong equivalence and $p$ is a simplicial finite-fibration. 
\end{theorem}

\begin{theorem}\label{new5.5}
    If $\phi:K\to L$ is a simplicial map and $p:M\to L$ is a simplicial fibration as in Theorem~\ref{Thm4}, then $\hsecat_m(\phi)=\hsecat_m(p)$.
\end{theorem}

\begin{proof}
    Suppose $\hsecat_m(p)=k$. Then there exist subcomplexes $L_0, \ldots, L_k$ of $L$ covering $L$ such that each $L_j$ has the property that for any simplicial map $\eta_j:P\to L_j$ from an $m$-dimensional simplicial complex $P$, there exists $s:P\to M$ such that $p\circ s\sim \iota_j\circ \eta_j$. 

    $\beta: K\to M$ is strong equivalence, i.e., there exists a simplicial map $\delta: M\to K$ such that $\delta\circ \beta \sim \id_K$ and $\beta\circ \delta\sim \id_M$. 

    Define a simplicial map $\overline{s}:P\to K$ by $\overline{s}=\delta\circ s$ so that we have

    \[
\begin{aligned}
\phi\circ \overline{s} 
    &= \phi \circ (\delta\circ s) \\
    &= (p\circ \beta)\circ (\delta\circ s) \\
    &= p\circ (\beta \circ \delta) \circ s \\
    &\sim p\circ \id_M \circ s\\
    &= p\circ s\\
    &=\iota_j \circ \eta_j.
\end{aligned}
\]

Therefore $\hsecat_m(\phi)\leq k$. 

On the other hand, suppose $\hsecat_m(\phi)=k$. Then there exist subcomplexes $L_0, \ldots, L_k$ of $L$ covering $L$ such that each $L_j$ has the property that for any simplicial map $\eta_j:P\to L_j$ from an $m$-dimensional simplicial complex $P$, there exists $\overline{s}:P\to K$ such that $p\circ \overline{s}\sim \iota_j\circ \eta_j$. 

If we define a simplicial map $s:P\to M$ by $s=\beta\circ \overline{s}$, we conclude that $\hsecat_m(p)\leq k$.
    
\end{proof}

\begin{theorem}
    For given simplicial complex $K$, $TC^{m}(K)=\secat_{m}(\Delta)$, where $\Delta:K\to K^{2}$ is the diagonal map. 
\end{theorem}

\begin{proof}
    By Example 2 in \cite{FTGCMVV}, there is a factorisation of the diagonal map $\Delta:K\to K\times K$ as in Theorem~\ref{Thm4}, that is, $\Delta = \pi \circ \beta$ where $\pi=(\alpha,\omega): PK\to K\times K$ is the path simplicial fibration and $\beta$ is a strong equivalence. If we apply Theorem~\ref{new5.5} on that factorisation, we complete the proof. 
\end{proof}


\section{Aspherical Spaces $K(\Gamma,1)$}

Let us recall that given a discrete group $\Gamma$, an aspherical space (Eilenberg-MacLane space) $B\Gamma=K(\Gamma,1)$ is defined to be a path-connected space such that $\pi_{1}(B\Gamma)=\Gamma$ and $\pi_{k}(B\Gamma)=0$ for all $k\neq 1.$ Since $B\Gamma$ is the unique up to homotopy equivalence, we define that LS-category of a discrete group to be LS-category of its aspherical space, i.e., $\cat(\Gamma):=\cat(B\Gamma)$. 

It is known that given a simplicial complex K, and taking its geometric realisation |K|, we obtain $\cat(|K|)\leq \scat(K)$ in \cite{FTMVMV2}.

\begin{theorem}\label{Aspherical Space}
    Let K, $|K|$ be a simplicial complex and its geometric realisation. If $|K|=B\Gamma$, then $$\cat(|K|)=\scat(K).$$ 
\end{theorem}

\begin{proof}
Since $\cat(B\Gamma)=cd(\Gamma)$ by \cite{EG}, the result is immediate when
$\Gamma$ has torsion, because then $cd(\Gamma)=\infty$. Hence assume that
$\Gamma$ is torsion-free.

Let $\cat(B\Gamma)=k.$ Then there exists an open categorical cover
$B\Gamma=\bigcup_{j=0}^{k}U_j,$ where each inclusion $\iota_j:U_j\hookrightarrow B\Gamma$ is null-homotopic.

After replacing $K$ by a sufficiently fine barycentric subdivision (which does not change $\scat(K)$), we may assume that this cover is refined by finite simplicial subcomplexes $L_j\subseteq K$ such that $|L_j|\subseteq U_j$ for every $j$. Since $\iota_j|_{|L_j|}$ is also null-homotopic, there exists a homotopy $H_j:|L_j|\times I\longrightarrow |K|$ between the inclusion and a constant map.

By the simplicial approximation theorem (after subdividing the triangulation
of $|L_j|\times I$ if necessary), $H_j$ is simplicially approximated by a
simplicial homotopy. Hence the inclusion $L_j\hookrightarrow K$ and the constant map $c_{v_0}$ are simplicially homotopic. By
\cite[Proposition~8(2)]{FTGCMVV}, simplicially homotopic simplicial maps lie
in the same contiguity class. Therefore the inclusion $L_j\hookrightarrow K$
is contiguous to a constant map, so each $L_j$ is categorical.

Thus the subcomplexes $L_0,\ldots,L_k$ form a categorical cover of $K$,
showing that $\scat(K)\le k=\cat(B\Gamma).$
Since the opposite inequality $\cat(|K|)\le \scat(K)$ holds in general, we conclude that $\cat(|K|)=\scat(K).$
\end{proof}

\begin{corollary}\label{corAspherical}
Let K, $|K|$ be a simplicial complex and its geometric realisation. If $|K|=B\Gamma$, then $$\cat_{m}(|K|)=\scat_{m}(K).$$ 
\end{corollary}

\begin{proof}
Since $\cat_{m}(B\Gamma)=\cat(B\Gamma)$ \cite{MVMLO} and $\cat_{m}(|K|)\leq \scat_{m}(K)$, the proof of the above theorem can be processed word by word modification since $\cat(B\Gamma)=cd(\Gamma)$ by \cite{EG}.
\end{proof}

Due to the homotopy invariance of aspherical spaces, there is a one-to-one correspondence between group homomorphism $\phi:\Gamma\to\Lambda$ and homotopy classes of maps $B\phi:B\Gamma\to B\Lambda$ that induce $\phi$ on the fundamental group. Hence, the following definition is well-defined.
\begin{definition}
 Let $\phi,\psi:\Gamma\to\Lambda$ be a homomorphism.
 \begin{enumerate}
\item The homotopic distance of $\phi, \psi$, $\D(\phi,\psi),$ is defined to be $\D(B\phi,B\psi).$
\item The m-homotopic distance of $\phi, \psi$, $\D_{m}(\phi,\psi),$ is defined to be $\D_{m}(B\phi,B\psi).$
\item The contiguity distance of $\phi, \psi$, $\SD(\phi,\psi),$ is defined to be $\SD(B\phi,B\psi).$
\item The m-contiguity distance of $\phi, \psi$, $\SD_{m}(\phi,\psi),$ is defined to be $\SD_{m}(B\phi,B\psi).$
 \end{enumerate}
\end{definition}

We prove the same statement of Theorem \ref{Aspherical Space} in the setting of homotopy and contiguity distances. We note that authors of \cite{MVMLO}  proved that $\D(\phi,\psi)=\D_{1}(\phi,\psi)$ where $\phi,\psi:X\to Y$ and $Y=B\Gamma$. 

\begin{theorem}\label{AspherMap}
Let $\phi,\psi:K\to L$ be simplicial maps with their geometric realisations $|\phi|,|\psi|:|K|\to |L|$. Let $|K|=B\Gamma$ and $|L|=B\Lambda$ be aspherical spaces. If $\Gamma,\Lambda$ are a torsion-free groups, then $$\D(|\phi|,|\psi|)=\SD(\phi,\psi).$$ 
\end{theorem}

\begin{proof}
Since $\D(|\phi|,|\psi|)\leq \SD(\phi,\psi),$ it suffices to prove the reverse inequality.

Assume that $\D(|\phi|,|\psi|)=k.$ Then there exist open subsets
$U_0,\ldots,U_k\subseteq |K|$ such that $|K|=\bigcup_{j=0}^{k}U_j,$ and for every $j$, the restrictions $|\phi|_{U_j},\,|\psi|_{U_j}:U_j\longrightarrow |L|$ are homotopic.

Since $\Gamma,\Lambda$ are a torsion-free groups, we assume $K$ and $L$ are finte simplicial complexes. 

After replacing $K$ by a sufficiently fine barycentric subdivision, we may
assume that the cover is refined by finite simplicial subcomplexes
$L_0,\ldots,L_k\subseteq K$ such that
$|L_j|\subseteq U_j$ for every $j$.

Since the restrictions $|\phi|_{|L_j|},\,|\psi|_{|L_j|}$
are homotopic, there exists a homotopy $H_j:|L_j|\times I\longrightarrow |L|$
between them.

By the Simplicial Approximation Theorem (after subdividing the triangulation
of $|L_j|\times I$, if necessary), the homotopy $H_j$ admits a simplicial
approximation. Consequently, the simplicial maps
$\phi|_{L_j},\,\psi|_{L_j}:L_j\longrightarrow L$ are simplicially homotopic.

Since each $L_{j}$ is finite, by Proposition~8(2) of \cite{FTGCMVV}, simplicially homotopic simplicial maps belong to the same contiguity class. Therefore, $\phi|_{L_j}\sim_c\psi|_{L_j}$ for every $j$.

Hence each subcomplex $L_j$ satisfies the defining property of the
contiguity distance. Since $K=\bigcup_{j=0}^{k}L_j,$ it follows that
$\SD(\phi,\psi)\leq k=\D(|\phi|,|\psi|).$

Combining this with the general inequality
$\D(|\phi|,|\psi|)\leq \SD(\phi,\psi),$ we conclude that
$\D(|\phi|,|\psi|)=\SD(\phi,\psi).$
\end{proof}

Similarly to Corollary \ref{corAspherical}, we state the following Corollary in the setting of $m$-homotopy and contiguity distance without proof since the proof is almost identical with above Theorem.

\begin{corollary}\label{AsphMap1}
Let $\phi,\psi:K\to L$ be simplicial maps with their geometric realisations $|\phi|,|\psi|:|K|\to |L|$. Let $|K|=B\Gamma$ and $|L|=B\Lambda$ be aspherical spaces. If $\Gamma,\Lambda$ are torsion-free groups, then $$\D_{m}(|\phi|,|\psi|)=\SD_{m}(\phi,\psi).$$     
\end{corollary}

Now we state the analog results of Theorem 2.19 in \cite{MVMLO} in the setting of contiguity distance. 

\begin{theorem}\label{SD1SD}
Let $\phi,\psi:K\to L$ be simplicial maps with their geometric realisations $|\phi|,|\psi|:|K|\to |L|$. Let $|K|=B\Gamma$ and $|L|=B\Lambda$ be aspherical spaces. If $\Gamma,\Lambda$ are a torsion-free group, then $$\SD_{1}(\phi,\psi)=\SD(\phi,\psi).$$ 
\end{theorem}

\begin{proof}
By Theorem \ref{AspherMap}, $\D(|\phi|,|\psi|)=\SD(\phi,\psi).$ By Corollary \ref{AsphMap1}, $\D_{1}(|\phi|,|\psi|)=\SD_{1}(\phi,\psi).$ By Theorem 2.19 \cite{MVMLO}, $\D(|\phi|,|\psi|)=\D_{1}(|\phi|,|\psi|).$ Combining all equality give the statement of Theorem. This completes the proof. 
\end{proof}

\begin{remark}
Alternative proof of above Theorem~\ref{SD1SD} can be given using ideas of Theorem 2.19 \cite{MVMLO}. In this paper, we want to follow the uniform approach meaning first give the proof for LS-category (simplicial LS-category) and then extended to new setting homotopy (contiguity) distances. Our proof also works when $\Lambda$ has a torsion element, but $\Gamma$ is a torsion-free group. In general case, where the torsion element of $\Gamma$ does not vanish under group homomorphisms, our proof does not work. Second author has the interesting results in the case of LS-category of group epimorphisms, $cat(\phi)=\D(\phi,1)$ where $1$ is a trivial homomorphism \cite{Ku1}.     
\end{remark}

\section{Acknowledgement} 

This paper has been submitted in partial fulfillment of the requirements for the PhD degree at Bursa Technical University. The second author was partially supported by the grant No. AP25796111 of the Science Committee of the Ministry of Science and Higher Education of the Republic of Kazakhstan. \\

On behalf of all authors, the corresponding author states that there is no conflict of interest.\\

The manuscript has no associated data.

\end{document}